\documentclass[10pt,a4paper,twoside]{amsart}

\usepackage{graphicx}
\usepackage{a4wide}

\newtheorem{definition}{Definition}[section]
\newtheorem{theorem}{Theorem}[section]
\newtheorem{lemma}{Lemma}[section]

\newtheorem*{maintheorem*}{Main Theorem}
\allowdisplaybreaks
\numberwithin{equation}{section}

\newcommand{\norm}[1]{\left\| #1 \right\|}

\newcommand{\eps}{\varepsilon}

\newcommand{\eb}{{\eps,\beta}}
\newcommand{\ueb}{u_\eb}

\newcommand{\pt}{\partial_t}

\newcommand{\px}{\partial_x }
\newcommand{\pxx}{\partial_{xx}^2}
\newcommand{\pxxx}{\partial_{xxx}^3}
\newcommand{\ptxxx}{\partial_{txxx}^4}
\newcommand{\pxxxx}{\partial_{xxxx}^4}

\newcommand{\ptxxxx}{\partial_{txxxx}^5}
\newcommand{\ptx}{\partial_{tx}^2}
\newcommand{\ptxx}{\partial_{txx}^3}

\renewcommand{\i}{\ifmmode\mathit{\mathchar"7010 }\else\char"10 \fi}
\renewcommand{\j}{\ifmmode\mathit{\mathchar"7011 }\else\char"11 \fi}
\newcommand{\R}{\mathbb{R}}
\newcommand{\N}{\mathbb{N}}

\newcommand{\supp}{\mathrm{supp}\,}

{%

\begin{enumerate}}%
{\end{enumerate}}

{%

\begin{enumerate}}%
{\end{enumerate}}

\begin{document}\large

\title[A Singular limit problem of Rosenau-KdV-RLW  type  ]{A singular limit problem\\ for the Rosenau-Korteweg-de Vries\\-regularized long wave\\ and Rosenau-Korteweg-de Vries  equation.}

\author[G. M. Coclite and L. di Ruvo]{Giuseppe Maria Coclite and Lorenzo di Ruvo}
\address[Giuseppe Maria Coclite and Lorenzo di Ruvo]
{\newline Department of Mathematics,   University of Bari, via E. Orabona 4, 70125 Bari,   Italy}
\email[]{giuseppemaria.coclite@uniba.it, lorenzo.diruvo@uniba.it}
\urladdr{http://www.dm.uniba.it/Members/coclitegm/}

\date{\today}



%35G25 Initial value problems for nonlinear higher-order PDE, nonlinear evolution equations
%35L05 Wave equation
%74S20 Finite difference methods
%35L65 Conservation laws
%65M12 Stability and convergence of numerical methods

\thanks{The authors are members of the Gruppo Nazionale per l'Analisi Matematica, la Probabilit\`a e le loro Applicazioni (GNAMPA) of the Istituto Nazionale di Alta Matematica (INdAM)}

\begin{abstract}
We consider the Rosenau-Korteweg-de Vries-regularized long wave and Rosenau-Korteweg-de Vries  equations, which contain nonlinear dispersive effects. We prove
that, as the diffusion parameter tends to zero, the solutions of the dispersive equations converge to the  unique entropy solution of a scalar conservation law.
The proof relies on deriving suitable a priori estimates together with an application of the compensated compactness method in the $L^p$ setting.
\end{abstract}

\maketitle

%\tableofcontents

\section{Introduction}\label{sec:intro}
The dynamics of shallow water waves that is observed along lake shores and beaches has been a research area for the past few decades in the area of
oceanography (see \cite{AB,ZZZC}). There are several models proposed in this context: Korteweg-de Vries (KdV) equation, Boussinesq equation, Peregrine equation, regularized long wave (RLW) equation, Kawahara equation, Benjamin-Bona-Mahoney equation, Bona-Chen equation and several others. These models are
derived from first principles under various different hypothesis and approximations. They are all well studied and very well understood.

The dynamics of dispersive shallow water waves, on the other hand, is captured with slightly different models like  Rosenau-
Kawahara, Rosenau-KdV, and Rosenau-KdV-RLW equations \cite{BTL,EMTYB,HXH,LB,RAB}.

In particular, the Rosenau-KdV-RLW equation is
\begin{equation}
\label{eq:RKV-1}
\pt u +a\px u +k\px u^{n}+b_1\pxxx u +b_2\ptxx u + c\ptxxxx u=0,\quad a,\,k,\,b_1,\,b_2,\,c\in\R.
\end{equation}
Here $u(t,x)$ is the nonlinear wave profile. The first term is the linear evolution one, while $a$ is the advection (or drifting) coefficient. The two dispersion coefficients are  $b_1$ and $b_2$ . The higher order dispersion coefficient is  $c$, while the coefficient of nonlinearity is $k$ where $n$ is the nonlinearity parameter. These are all known and given parameters.

In \cite{RAB}, the authors analyzed \eqref{eq:RKV-1}. They got solitary waves, shock waves and singular solitons along with conservation laws.

In the case $n=2,\, a=0,\, k=1,\, b_1=1,\, b_2=-1,\, c=1$, we have
\begin{equation}
\label{eq:RKV-23}
\pt u +\px u^2 +\pxxx u -\ptxx u +\ptxxxx u=0.
\end{equation}
Choosing $n=2,\, a=0,\, k=1,\, b_2=b_1=0,\, c=1$, \eqref{eq:RKV-1} reads
\begin{equation}
\label{eq:RKV-2}
\pt u + \px u^{2} + \ptxxxx u=0,
\end{equation}
which is known as Rosenau equation (see \cite{Ro1,Ro2}). Existence and uniqueness of solutions for \eqref{eq:RKV-2} has been proved in \cite{P}.

Finally, if $n=2,\, a=0,\, k=1,\, b_1=1,\, b_2=0,\, c=1$, \eqref{eq:RKV-1} reads
\begin{equation}
\label{eq:RKV-3}
\pt u  +\px u^{2}+\pxxx u + \ptxxxx u=0,
\end{equation}
which is known as Rosenau-KdV equation.

In \cite{Z}, the author discussed the solitary wave solutions and \eqref{eq:RKV-3}. In \cite{HXH}, a conservative linear finite difference scheme for the numerical solution for an initial-boundary value problem of Rosenau-KdV equation was considered.
In \cite{E,RTB}, the authors discussed the solitary solutions for \eqref{eq:RKV-3} with  solitary ansatz method. The authors also gave  two invariants for \eqref{eq:RKV-3}. In particular, in \cite{RTB}, the authors  studied the two types of soliton solutions, one is a solitary wave  and the other is a singular soliton. In \cite{ZZ}, the authors proposed an average linear finite difference scheme for the numerical solution of the initial-boundary value problem for  \eqref{eq:RKV-3}.

If $n=2, \, a=0,\, k=1,\, b_1=0,\, b_2=-1,\, c=1$, \eqref{eq:RKV-1} reads
\begin{equation}
\label{eq:RKV-30}
\pt u +\px u^2 -\ptxx u +\ptxxxx u=0,
\end{equation}
which is the Rosenau-RLW equation.

In this paper, we analyze \eqref{eq:RKV-23} and \eqref{eq:RKV-30}. Arguing as \cite{CdREM}, we re-scale the equations as follows
\begin{align}
\label{eq:RKV}
\pt u+ \px u^2+\beta\pxxx u - \beta\ptxx u +\beta^2\ptxxxx u=&0, \\
\label{eq:RKV34}
\pt u+ \px u^2 -\beta\ptxx u +\beta^2\ptxxxx u=&0,
\end{align}
where $\beta$ is the diffusion parameter.

We are interested in the no high frequency limit,  we send $\beta\to 0$ in \eqref{eq:RKV} and \eqref{eq:RKV34}. In this way we pass from \eqref{eq:RKV} and \eqref{eq:RKV34} to the equation
\begin{equation}
\label{eq:BU}
\pt u+\px u^2=0
\end{equation}
which is a scalar conservation law.
We prove that, when $\beta\to0$, the solutions of \eqref{eq:RKV} and \eqref{eq:RKV34} converge  to the unique entropy solution \eqref{eq:BU}.

The paper is organized in three sections. In Section \ref{sec:vv}, we prove the convergence of \eqref{eq:RKV} to \eqref{eq:BU}, while in Section \ref{sec:Ro1}, we show how to modify
the argument of Section \ref{sec:vv} and prove the convergence of \eqref{eq:RKV34} to \eqref{eq:BU}.

\section{The Rosenau-KdV-RLW equation.}\label{sec:vv}

In this section, we consider \eqref{eq:RKV} and augment it with the initial condition
\begin{equation}
u(0,x)=u_{0}(x),
\end{equation}
on which we assume that
\begin{equation}
\label{eq:assinit}
u_0\in L^2(\R)\cap L^4(\R).
\end{equation}
We study the dispersion-diffusion limit for \eqref{eq:RKV}. Therefore, we consider the following fifth order approximation
\begin{equation}
\label{eq:RKV-eps-beta}
\begin{cases}
\pt\ueb+ \px \ueb^2 + \beta\pxxx \ueb - \beta\ptxx \ueb\\
\qquad\quad\qquad\quad +\beta^2\ptxxxx\ueb=\eps\pxx\ueb, &\qquad t>0, \ x\in\R ,\\
\ueb(0,x)=u_{\eps,\beta,0}(x), &\qquad x\in\R,
\end{cases}
\end{equation}
where $u_{\eps,\beta,0}$ is a $C^\infty$ approximation of $u_{0}$ such that
\begin{equation}
\begin{split}
\label{eq:u0eps}
&u_{\eps,\,\beta,\,0} \to u_{0} \quad  \textrm{in $L^{p}_{loc}(\R)$, $1\le p < 4$, as $\eps,\,\beta \to 0$,}\\
&\norm{u_{\eps,\beta, 0}}^2_{L^2(\R)}+\norm{u_{\eps,\beta, 0}}^4_{L^4(\R)}+(\beta+ \eps^2) \norm{\px u_{\eps,\beta,0}}^2_{L^2(\R)}\le C_0,\quad \eps,\beta >0,  \\
&(\beta\eps + \beta\eps^2 +\beta^2)\norm{\pxx u_{\eps,\beta,0}}^2_{L^2(\R)}+(\beta^2\eps^2+\beta^3)\norm{\pxxx u_{\eps,\beta,0}}^2_{L^2(\R)} \le C_{0}, \quad \eps,\beta >0,\\
&\beta^4\norm{\pxxxx u_{\eps,\beta,0}}^2_{L^2(\R)}\le C_{0}, \quad \eps,\beta >0,
\end{split}
\end{equation}
and $C_0$ is a constant independent on $\eps$ and $\beta$.

The main result of this section is the following theorem.
\begin{theorem}
\label{th:main}
Assume that \eqref{eq:assinit} and  \eqref{eq:u0eps} hold.
If
\begin{equation}
\label{eq:beta-eps}
\beta=\mathbf{\mathcal{O}}(\eps^{4}),
\end{equation}
then, there exist two sequences $\{\eps_{n}\}_{n\in\N}$, $\{\beta_{n}\}_{n\in\N}$, with $\eps_n, \beta_n \to 0$, and a limit function
\begin{equation*}
u\in L^{\infty}(\R^{+}; L^2(\R)\cap L^{4}(\R)),
\end{equation*}
such that
\begin{align}
\label{eq:con-u}
&u_{\eps_n, \beta_n}\to u \quad \text{strongly in $L^{p}_{loc}(\R^{+}\times\R)$, for each $1\le p <4$};\\
\label{eq:entropy1}
&u \quad\text{is the unique entropy solution of \eqref{eq:BU}}.
\end{align}
\end{theorem}

Let us prove some a priori estimates on $\ueb$, denoting with $C_0$ the constants which depend only on the initial data.
\begin{lemma}\label{lm:l-2}
For each $t>0$,
\begin{equation}
\label{eq:l-2}
\begin{split}
\norm{\ueb(t,\cdot)}^2_{L^2(\R)}&+\beta\norm{\px\ueb(t,\cdot)}^2_{L^2(\R)}\\
&+\beta^2\norm{\pxx\ueb(t,\cdot)}^2_{L^2(\R)}+2\eps\int_{0}^{t}\norm{\px\ueb(s,\cdot)}^2_{L^2(\R)}ds\le C_{0}.
\end{split}
\end{equation}
In particular, we have
\begin{align}
\label{eq:u-l-infty}
\norm{\ueb(t,\cdot)}_{L^{\infty}(\R)}\le& C_{0}\beta^{-\frac{1}{4}},\\
\label{eq:ux-l-infty}
\norm{\px\ueb(t,\cdot)}_{L^{\infty}(\R)}\le& C_{0}\beta^{-\frac{3}{4}}.
\end{align}
\end{lemma}
\begin{proof}
Multiplying \eqref{eq:RKV-eps-beta} by $\ueb$, we have
\begin{equation}
\label{eq:ks4}
\begin{split}
\ueb\pt\ueb &+ 2\ueb^2\px \ueb+  \beta\ueb\pxxx \ueb\\
&-\beta\ueb\ptxx\ueb+\beta^2\ueb\ptxxxx\ueb=\eps\ueb\pxx\ueb.
\end{split}
\end{equation}
Since
\begin{align*}
\int_{\R}\ueb\pt\ueb dx=&\frac{1}{2}\frac{d}{dt}\norm{\ueb(t,\cdot)}^2_{L^2(\R)},\\
2\int_{\R}\ueb^2\px\ueb dx=&0,\\
\beta\int_{\R}\ueb\pxxx \ueb dx= &\int_{\R}\px\ueb\pxx\ueb dx=0,\\
-\beta\int_{\R}\ueb\ptxx\ueb dx=&\frac{\beta}{2}\frac{d}{dt}\norm{\px\ueb(t,\cdot)}^2_{L^2(\R)},\\
\beta^2\int_{\R}\ueb\ptxxxx\ueb dx=& -\beta^2\int_{\R}\px\ueb\ptxxx\ueb dx \\
=& \frac{\beta^2}{2}\frac{d}{dt}\norm{\pxx\ueb(t,\cdot)}^2_{L^2(\R)},\\
\eps\int_{\R}\ueb\pxx\ueb dx=& -\eps\norm{\px\ueb(t,\cdot)}^2_{L^2(\R)}.
\end{align*}
Integrating \eqref{eq:ks4} on $\R$, we get
\begin{equation}
\label{eq:ks6}
\begin{split}
\frac{d}{dt}\norm{\ueb(t,\cdot)}^2_{L^2(\R)} &+ \beta\frac{d}{dt}\norm{\px\ueb(t,\cdot)}^2_{L^2(\R)}\\
&+ \beta^2\frac{d}{dt}\norm{\pxx\ueb(t,\cdot)}^2_{L^2(\R)} +2\eps\norm{\px\ueb(t,\cdot)}^2_{L^2(\R)} =0.
\end{split}
\end{equation}

\eqref{eq:l-2} follows from \eqref{eq:u0eps}, \eqref{eq:ks6} and an integration on $(0,t)$.

We prove \eqref{eq:u-l-infty}. Due to \eqref{eq:l-2} and the H\"older inequality,
\begin{align*}
\ueb^2(t,x)=&2\int_{-\infty}^x \ueb\px\ueb dx \le 2\int_{\R}\vert \ueb\vert \vert \px\ueb\vert dx\\
\le&2\norm{\ueb(t,\cdot)}_{L^2(\R)}\norm{\px\ueb(t,\cdot)}_{L^2(\R)}\le C_{0}\beta^{-\frac{1}{2}}.
\end{align*}
Therefore,
\begin{equation*}
\vert \ueb(t,x)\vert \le C_{0}\beta^{-\frac{1}{4}},
\end{equation*}
which gives \eqref{eq:u-l-infty}.

Finally, we prove \eqref{eq:ux-l-infty}. Thanks to \eqref{eq:l-2} and the H\"older inequality,
\begin{align*}
\px\ueb^2(t,x)=&2\int_{-\infty}^x \px\ueb\pxx\ueb dx \le 2\int_{\R}\vert \px\ueb\vert \vert \pxx\ueb\vert dx\\
\le&2\norm{\px\ueb(t,\cdot)}_{L^2(\R)}\norm{\pxx\ueb(t,\cdot)}_{L^2(\R)}\le C_{0}\beta^{-\frac{1}{2}}C_{0}\beta^{-1}\le C_{0}\beta^{-\frac{3}{2}}.
\end{align*}
Hence,
\begin{equation*}
\norm{\px\ueb(t,\cdot)}_{L^{\infty}(\R)} \le C_{0}\beta^{-\frac{3}{4}},
\end{equation*}
that is \eqref{eq:ux-l-infty}.
\end{proof}
Following \cite[Lemma $2.2$]{Cd}, or \cite[Lemma $2.2$]{CdREM}, or \cite[Lemma $4.2$]{CK}, we prove the following result.
\begin{lemma}\label{lm:ux-l-2}
Assume \eqref{eq:beta-eps}. For each $t>0$,
\begin{itemize}
\item[$i)$] the family $\{\ueb\}_{\eps,\beta}$ is bounded in $L^{\infty}(\R^{+};L^{4}(\R))$;
\item[$ii)$]the families $\{\eps\px\ueb\}_{\eps,\beta},\, \{\sqrt{\beta\eps}\pxx\ueb\}_{\eps,\beta}, \,  \{\eps\sqrt{\beta}\pxx\ueb\}_{\eps,\beta},\,\{\eps\beta\pxxx\ueb\}_{\eps,\beta}$,\\
  $\{\beta\sqrt{\beta}\pxxx\ueb\}_{\eps,\beta},\,\{\beta\pxxxx\ueb\}_{\eps,\beta}$  are bounded in $L^{\infty}(\R^{+};L^2(\R))$;
\item[$iii)$] the families $\{\sqrt{\beta\eps}\ptx\ueb\}_{\eps,\beta},\,\{\beta\sqrt{\eps}\ptxx\ueb\}_{\eps,\beta},\,\{\beta\sqrt{\beta\eps}\ptxxx\ueb\}_{\eps,\beta},$\\
    $\{\eps\beta\sqrt{\beta}\pxxx\ueb \}_{\eps,\beta},\,\{\sqrt{\eps}\ueb\px\ueb\}_{\eps,\beta}, \, \{\eps\sqrt{\eps}\pxx\ueb \}_{\eps,\beta} $ are bounded in\\ $L^2(\R^{+}\times\R)$.
\end{itemize}
Moreover,
\begin{align}
\label{eq:ux-uxx}
\beta\int_0^t\norm{\px\ueb(s,\cdot)\pxx\ueb(s,\cdot)}_{L^1(\R)}ds\le& C_0\eps^2, \quad t>0,\\
\label{eq:uxx-l-2}
\beta^2\int_{0}^t\norm{\pxx\ueb(t,\cdot)}^2_{L^2(\R)}ds \le &C_0\eps^5, \quad t>0.
\end{align}
\end{lemma}
\begin{proof}
Let $A,\,B,\,C$ be some positive constants which will be specified later. Multiplying  \eqref{eq:RKV-eps-beta} by
\begin{equation*}
\ueb^3 -A\beta\eps\ptxx \ueb -B\eps^2\pxx\ueb+C\beta^2\pxxxx\ueb,
\end{equation*}
we have
\begin{equation}
\label{eq:KSmp}
\begin{split}
&\left(\ueb^3 -A\beta\eps\ptxx \ueb -B\eps^2\pxx\ueb+C\beta^2\pxxxx\ueb\right)\pt\ueb\\
&\qquad\quad +2\left(\ueb^3 -A\beta\eps\ptxx \ueb -B\eps^2\pxx\ueb+C\beta^2\pxxxx\ueb\right) \ueb\px\ueb\\
&\qquad\quad +\beta\left(\ueb^3 -A\beta\eps\ptxx \ueb -B\eps^2\pxx\ueb+C\beta^2\pxxxx\ueb\right)\pxxx\ueb\\
&\qquad\quad -\beta\left(\ueb^3 -A\beta\eps\ptxx \ueb -B\eps^2\pxx\ueb+C\beta^2\pxxxx\ueb\right)\ptxx\ueb\\
&\qquad\quad+\beta^2\left(\ueb^3 -A\beta\eps\ptxx \ueb -B\eps^2\pxx\ueb+C\beta^2\pxxxx\ueb\right)\ptxxxx\ueb\\
&\qquad=\eps\left(\ueb^3 -A\beta\eps\ptxx \ueb -B\eps^2\pxx\ueb+C\beta^2\pxxxx\ueb\right)\pxx\ueb.
\end{split}
\end{equation}
We observe that
\begin{equation}
\label{eq:p110}
\begin{split}
&\int_{\R}\left(\ueb^3 -A\beta\eps\ptxx \ueb -B\eps^2\pxx\ueb+C\beta^2\pxxxx\ueb\right)\pt\ueb dx\\
&\qquad= \frac{1}{4}\frac{d}{dt}\norm{\ueb(t,\cdot)}^4_{L^{4}(\R)} +A\beta\eps\norm{\ptx\ueb(t,\cdot)}^2_{L^2(\R)}\\
&\qquad \quad +\frac{B\eps^2}{2}\frac{d}{dt}\norm{\px\ueb(t,\cdot)}^2_{L^2(\R)}+\frac{C\beta^2}{2}\frac{d}{dt}\norm{\pxx\ueb(t,\cdot)}^2_{L^2(\R)}.
\end{split}
\end{equation}
We have that
\begin{equation}
\label{eq:p-100}
\begin{split}
&2\int_{\R}\left(\ueb^3 -A\beta\eps\ptxx \ueb -B\eps^2\pxx\ueb+C\beta^2\pxxxx\ueb\right) \ueb\px\ueb dx\\
&\qquad =-2A\beta\eps\int_{\R} \ueb\px\ueb\ptxx \ueb dx -2B\eps^2 \int_{\R}\ueb\px\ueb\pxx\ueb dx \\
&\qquad\quad -2C\beta^2\int_{\R}(\px\ueb)^2\pxxx\ueb dx -2C\beta^2\int_{\R} \ueb\pxx\ueb\pxxx\ueb dx.
\end{split}
\end{equation}
Since
\begin{equation}
\label{eq:p101}
\begin{split}
-2C\beta^2\int_{\R}(\px\ueb)^2\pxxx\ueb &-2C\beta^2\int_{\R}\ueb\pxx\ueb\pxxx\ueb dx\\
=&5C\beta^2\int_{\R}(\pxx\ueb)^2\px\ueb dx\\
=&-\frac{5\beta^2}{2}\int_{\R}(\px\ueb)^2\pxxx\ueb dx,
\end{split}
\end{equation}
it follows from \eqref{eq:p-100} and \eqref{eq:p101} that
\begin{equation}
\label{eq:p102}
\begin{split}
&2\int_{\R}\left(\ueb^3 -A\beta\eps\ptxx \ueb -B\eps^2\pxx\ueb+C\beta^2\pxxxx\ueb\right) \ueb\px\ueb dx\\
&\qquad =-2A\beta\eps\int_{\R} \ueb\px\ueb\ptxx \ueb dx -2B\eps^2 \int_{\R}\ueb\px\ueb\pxx\ueb dx \\
&\qquad\quad -\frac{5C\beta^2}{2}\int_{\R}(\px\ueb)^2\pxxx\ueb dx\\.
\end{split}
\end{equation}
We observe
\begin{equation}
\label{eq:p111}
\begin{split}
&\beta\int_{\R}\left(\ueb^3 -A\beta\eps\ptxx \ueb -B\eps^2\pxx\ueb+C\beta^2\pxxxx\ueb\right)\pxxx\ueb dx\\
&\qquad\quad= -3\beta\int_{\R}\ueb^2\px\ueb\pxx\ueb dx -A\beta^2\eps\int_{\R}\pxxx\ueb\ptxx \ueb dx.
\end{split}
\end{equation}
We get
\begin{equation}
\label{eq:p112}
\begin{split}
&-\beta\int_{\R}\left(\ueb^3 -A\beta\eps\ptxx \ueb -B\eps^2\pxx\ueb+C\beta^2\pxxxx\ueb\right)\ptxx\ueb dx\\
&\qquad=3\beta\int_{\R}\ueb^2\px\ueb\ptx\ueb dx +A\beta^2\eps\norm{\ptxx\ueb(t,\cdot)}^2_{L^2(\R)}\\
&\qquad\quad+\frac{B\beta\eps^2}{2}\frac{d}{dt}\norm{\pxx\ueb(t,\cdot)}^2_{L^2(\R)}+\frac{C\beta^3}{2}\frac{d}{dt}\norm{\pxxx\ueb(t,\cdot)}^2_{L^2(\R)}.
\end{split}
\end{equation}
We have that
\begin{equation}
\label{eq:p113}
\begin{split}
&\beta^2\int_{\R}\left(\ueb^3 -A\beta\eps\ptxx \ueb -B\eps^2\pxx\ueb+C\beta^2\pxxxx\ueb\right)\ptxxxx\ueb dx\\
&\qquad= -3\beta^2\int_{\R}\ueb^2\px\ueb\ptxxx\ueb dx +A\beta^3\eps\norm{\ptxxx\ueb(t,\cdot)}^2_{L^2(\R)}\\
&\qquad\quad +\frac{B\beta^2\eps^2}{2}\frac{d}{dt}\norm{\pxxx\ueb(t,\cdot)}^2_{L^2(\R)}+\frac{C\beta^4}{2}\frac{d}{dt}\norm{\pxxxx\ueb(t,\cdot)}^2_{L^2(\R)}.
\end{split}
\end{equation}
Moreover,
\begin{equation}
\label{eq:p114}
\begin{split}
&\eps\int_{\R}\left(\ueb^3 -A\beta\eps\ptxx \ueb -B\eps^2\pxx\ueb+C\beta^2\pxxxx\ueb\right)\pxx\ueb dx\\
&\qquad= -3\eps\norm{\ueb(t,\cdot)\px\ueb(t,\cdot)}^2_{L^2(\R)} -\frac{A\beta\eps}{2}\frac{d}{dt}\norm{\pxx\ueb(t,\cdot)}^2_{L^2(\R)}\\
&\qquad \quad -\eps^3B\norm{\pxx\ueb(t,\cdot)}^2_{L^2(\R)}-\beta^2\eps C\norm{\pxxx\ueb(t,\cdot)}^2_{L^2(\R)}.
\end{split}
\end{equation}
It follows from \eqref{eq:p110}, \eqref{eq:p102}, \eqref{eq:p111}, \eqref{eq:p112}, \eqref{eq:p113}, \eqref{eq:p114}, and an integration of \eqref{eq:KSmp} on $\R$ that
\begin{equation}
\label{eq:p120}
\begin{split}
&\frac{d}{dt}\left(\frac{1}{4}\norm{\ueb(t,\cdot)}^4_{L^{4}(\R)} +\frac{\left(A\beta\eps+B\beta\eps^2+C\beta^2\right)}{2}\norm{\pxx\ueb(t,\cdot)}^2_{L^2(\R)}       \right)\\
&\qquad\quad +\frac{d}{dt}\left(\frac{B\eps^2}{2}\norm{\px\ueb(t,\cdot)}^2_{L^2(\R)}+\frac{\left(B\beta^2\eps^2+C\beta^3\right)}{2}\norm{\pxxx\ueb(t,\cdot)}^2_{L^2(\R)}\right)\\ &\qquad\quad +\frac{C\beta^4}{2}\frac{d}{dt}\norm{\pxxxx\ueb(t,\cdot)}^2_{L^2(\R)}+A\beta\eps\norm{\ptx\ueb(t,\cdot)}_{L^2(\R)}\\
&\qquad\quad +A\beta^2\eps\norm{\ptxx\ueb(t,\cdot)}^2_{L^2(\R)}+A\beta^3\eps\norm{\ptxxx\ueb(t,\cdot)}^2_{L^2(\R)}\\
&\qquad\quad +3\eps\norm{\ueb(t,\cdot)\px\ueb(t,\cdot)}^2_{L^2(\R)}+\eps^3B\norm{\pxx\ueb(t,\cdot)}^2_{L^2(\R)}\\
&\qquad\quad +\beta^2\eps C\norm{\pxxx\ueb(t,\cdot)}^2_{L^2(\R)}\\
&\qquad = 2A\beta\eps\int_{\R}\ueb\px\ueb\ptxx\ueb dx +2B\eps^2\int_{\R}\ueb\px\ueb\pxx\ueb dx\\
&\qquad\quad \frac{5C\beta^2}{2}\int_{\R}(\px\ueb)^2\pxxx\ueb dx -3\beta\int_{\R}\ueb^2\px\ueb\pxx\ueb dx\\
&\qquad\quad +A\beta^2\eps\int_{\R}\pxxx\ueb\ptxx\ueb - 3\beta \int_{\R}\ueb^2\px\ueb\ptx\ueb dx\\
&\qquad\quad +3\beta^2\int_{\R}\ueb^2\px\ueb\ptxxx\ueb dx.
\end{split}
\end{equation}
Due to the Young inequality,
\begin{equation}
\label{eq:You23}
\begin{split}
&2A\beta\eps\left\vert \int_{\R}\ueb\px\ueb\ptxx\ueb dx \right\vert\le A\eps\int_{\R}\left\vert 2 \ueb\px\ueb \right\vert \vert \beta \ptxx\ueb\vert dx\\
&\qquad\le 2A\eps\norm{\ueb(t,\cdot)\px\ueb(t,\cdot)}^2_{L^2(\R)}+\frac{A\beta^2\eps}{2}\norm{\ptxx\ueb(t,\cdot)}^2_{L^2(\R)},\\
&2B\eps^2\left\vert\int_{\R}\ueb\px\ueb\pxx\ueb dx\right\vert \le \int_{\R}\left\vert \eps^{\frac{1}{2}}\ueb\px\ueb\right\vert \left\vert\eps^{\frac{3}{2}} 2B\pxx\ueb\right\vert dx\\
&\qquad \le \frac{\eps}{2}\norm{\ueb(t,\cdot)\px\ueb(t,\cdot)}^2_{L^2(\R)} +2B^2\eps^3 \norm{\pxx\ueb(t,\cdot)}^2_{L^2(\R)}.
\end{split}
\end{equation}
Hence, from \eqref{eq:p120},
\begin{equation}
\label{eq:237}
\begin{split}
&\frac{d}{dt}\left(\frac{1}{4}\norm{\ueb(t,\cdot)}^4_{L^{4}(\R)} +\frac{\left(A\beta\eps+B\beta\eps^2+C\beta^2\right)}{2}\norm{\pxx\ueb(t,\cdot)}^2_{L^2(\R)}       \right)\\
&\qquad\quad +\frac{d}{dt}\left(\frac{B\eps^2}{2}\norm{\px\ueb(t,\cdot)}^2_{L^2(\R)}+\frac{\left(B\beta^2\eps^2+C\beta^3\right)}{2}\norm{\pxxx\ueb(t,\cdot)}^2_{L^2(\R)}\right)\\ &\qquad\quad +\frac{C\beta^4}{2}\frac{d}{dt}\norm{\pxxxx\ueb(t,\cdot)}^2_{L^2(\R)}+A\beta\eps\norm{\ptx\ueb(t,\cdot)}^2_{L^2(\R)}\\
&\qquad\quad +\frac{A\beta^2\eps}{2}\norm{\ptxx\ueb(t,\cdot)}^2_{L^2(\R)}+A\beta^3\eps\norm{\ptxxx\ueb(t,\cdot)}^2_{L^2(\R)}\\
&\qquad\quad +\left(\frac{5}{2} -2A\right)\eps\norm{\ueb(t,\cdot)\px\ueb(t,\cdot)}^2_{L^2(\R)}+\beta^2\eps C\norm{\pxxx\ueb(t,\cdot)}^2_{L^2(\R)}\\
&\qquad\quad +\left(B-2B^2\right)\eps^3\norm{\pxx\ueb(t,\cdot)}^2_{L^2(\R)}\\
&\qquad \le \frac{5C\beta^2}{2}\int_{\R}(\px\ueb)^2\vert\pxxx\ueb\vert dx +3\beta\int_{\R} \ueb^2\vert\px\ueb\vert\pxx\ueb\vert dx\\
&\qquad\quad +A \beta^2\eps \int_{\R}\vert\pxxx\ueb\vert\vert\ptxx\vert dx +3\beta\int_{\R}\ueb^2\vert\px\ueb\vert\vert\ptx\ueb\vert dx\\
&\qquad\quad +3\beta^2\int_{\R} \ueb^2 \vert \px\ueb\vert \vert\ptxxx\ueb\vert dx.
\end{split}
\end{equation}
From \eqref{eq:beta-eps}, we have
\begin{equation}
\label{eq:beta-eps-1}
\beta\le D^2\eps^4,
\end{equation}
where $D$ is a positive constant which will be specified later. It follows from \eqref{eq:ux-l-infty}, \eqref{eq:beta-eps-1} and the Young inequality  that
\begin{equation}
\label{eq:p255}
\begin{split}
&\frac{5C\beta^2}{2}\int_{\R}(\px\ueb)^2\vert\pxxx\ueb\vert dx= \beta^2C\int_{\R}\frac{5}{2\eps^\frac{1}{2}}(\px\ueb)^2\left\vert \eps^{\frac{1}{2}}\pxxx\ueb\right\vert\\
&\qquad \le \frac{25C\beta^2}{8\eps}\int_{\R}(\px\ueb)^4 dx + \frac{C\beta^2\eps}{2}\norm{\pxxx\ueb(t,\cdot)}^2_{L^2(\R)}\\
&\qquad \le \frac{25C}{8\eps}\beta^2\norm{\px\ueb(t,\cdot)}^2_{L^{\infty}(\R)}\norm{\px\ueb(t,\cdot)}^2_{L^2(\R)} + \frac{C\beta^2\eps}{2}\norm{\pxxx\ueb(t,\cdot)}^2_{L^2(\R)}\\
&\qquad \le \frac{C_{0}\beta^{\frac{1}{2}}}{\eps}\norm{\px\ueb(t,\cdot)}^2_{L^2(\R)} + \frac{C\beta^2\eps}{2}\norm{\pxxx\ueb(t,\cdot)}^2_{L^2(\R)}\\
&\qquad \le C_0D\eps \norm{\px\ueb(t,\cdot)}^2_{L^2(\R)} + \frac{C\beta^2\eps}{2}\norm{\pxxx\ueb(t,\cdot)}^2_{L^2(\R)}.
\end{split}
\end{equation}
Due to \eqref{eq:u-l-infty}, \eqref{eq:beta-eps-1} and the Young inequality,
\begin{equation}
\label{eq:p256}
\begin{split}
&3\beta\int_{\R} \ueb^2\vert\px\ueb\vert\pxx\ueb\vert dx\le 3\beta \norm{\ueb(t,\cdot)}^2_{L^{\infty}(\R)}\int_{\R}\vert\px\ueb\vert\pxx\ueb\vert dx\\
&\qquad\le C_{0}\beta^{\frac{1}{2}}\int_{\R}\vert\px\ueb\vert\vert\pxx\ueb\vert dx\le\int_{\R}\left\vert \eps^{\frac{1}{2}}\px\ueb\right\vert\left\vert C_0 D\eps^{\frac{3}{2}}\pxx\ueb\right\vert dx\\
&\qquad \le \eps \norm{\px\ueb(t,\cdot)}^2_{L^2(\R)} +D^2C^2_{0}\eps^3\norm{\pxx\ueb(t,\cdot)}^2_{L^2(\R)}.
\end{split}
\end{equation}
Thanks to the Young inequality,
\begin{equation}
\label{eq:p257}
\begin{split}
&A \beta^2\eps \int_{\R}\vert\pxxx\ueb\vert\vert\ptxx\ueb\vert dx=A \beta^2\eps \int_{\R}\left\vert 2\pxxx\ueb\right\vert\left\vert\frac{1}{2}\ptxx\ueb\right\vert dx\\
&\qquad\le 2A\beta^2\eps\norm{\pxxx\ueb(t,\cdot)}^2_{L^2(\R)}+\frac{A\beta^2\eps}{8}\norm{\ptxx\ueb(t,\cdot)}^2_{L^2(\R)}.
\end{split}
\end{equation}
It follows from \eqref{eq:u-l-infty}, \eqref{eq:beta-eps-1} and the Young inequality that
\begin{equation}
\label{eq:p289}
\begin{split}
&3\beta\int_{\R}\ueb^2\vert\px\ueb\vert\vert\ptx\ueb\vert dx = \beta\int_{\R}\left\vert\frac{3\ueb^2\px\ueb}{\eps^{\frac{1}{2}}A^{\frac{1}{2}}}\right\vert\left\vert\eps^{\frac{1}{2}}A^{\frac{1}{2}}\ptx\ueb\right\vert dx\\
&\qquad\le\frac{9\beta}{2\eps A}\int_{\R}\ueb^4(\px\ueb)^2 dx + \frac{\beta\eps A}{2}\norm{\ptx\ueb(t,\cdot)}^2_{L^2(\R)}\\
&\qquad\le \frac{9}{2\eps A}\beta\norm{\ueb(t,\cdot)}^2_{L^{\infty}(\R)}\norm{\ueb(t,\cdot)\px\ueb(t,\cdot)}^2_{L^2(\R)}\\
&\qquad\quad + \frac{\beta\eps A}{2}\norm{\ptx\ueb(t,\cdot)}^2_{L^2(\R)}\\
&\qquad \le\frac{C_{0}\beta^{\frac{1}{2}}}{\eps A}\norm{\ueb(t,\cdot)\px\ueb(t,\cdot)}^2_{L^2(\R)}+ \frac{\beta\eps A}{2}\norm{\ptx\ueb(t,\cdot)}^2_{L^2(\R)}\\    &\qquad \le \frac{C_0 D}{A}\eps\norm{\ueb(t,\cdot)\px\ueb(t,\cdot)}^2_{L^2(\R)}+ \frac{\beta\eps A}{2}\norm{\ptx\ueb(t,\cdot)}^2_{L^2(\R)}.
\end{split}
\end{equation}
Again by \eqref{eq:u-l-infty}, \eqref{eq:beta-eps-1} and the Young inequality,
\begin{equation}
\label{eq:290}
\begin{split}
&3\beta^2\int_{\R} \ueb^2 \vert \px\ueb\vert \vert\ptxxx\ueb\vert dx =\int_{\R}\left\vert\frac{3\beta^{\frac{1}{2}}\ueb^2\px\ueb}{\eps^{\frac{1}{2}}A^{\frac{1}{2}}}\right\vert \left \vert\beta^{\frac{3}{2}}\eps^{\frac{1}{2}}A^{\frac{1}{2}}\ptxxx\ueb\right\vert dx\\
&\qquad \le \frac{9\beta}{2\eps A}\int_{\R}\ueb^4(\px\ueb)^2 dx + \frac{\beta^3\eps A}{2} \norm{\ptxxx\ueb(t,\cdot)}^2_{L^2(\R)}\\
&\qquad \le \frac{9}{2\eps A}\beta\norm{\ueb(t,\cdot)}^2_{L^{\infty}(\R)}\norm{\ueb(t,\cdot)\px\ueb(t,\cdot)}^2_{L^2(\R)}\\
&\qquad\quad +\frac{\beta^3\eps A}{2} \norm{\ptxxx\ueb(t,\cdot)}^2_{L^2(\R)}\\
&\qquad \le \frac{C_{0}\beta^{\frac{1}{2}}}{\eps A}\norm{\ueb(t,\cdot)\px\ueb(t,\cdot)}^2_{L^2(\R)}+ \frac{\beta^3\eps A}{2} \norm{\ptxxx\ueb(t,\cdot)}^2_{L^2(\R)}\\
&\qquad \le \frac{C_{0}D}{A}\eps\norm{\ueb(t,\cdot)\px\ueb(t,\cdot)}^2_{L^2(\R)}+ \frac{\beta^3\eps A}{2} \norm{\ptxxx\ueb(t,\cdot)}^2_{L^2(\R)}.
\end{split}
\end{equation}
From \eqref{eq:237}, \eqref{eq:p255}, \eqref{eq:p256}, \eqref{eq:p257}, \eqref{eq:p289} and \eqref{eq:290}, we get
\begin{equation}
\label{eq:p300}
\begin{split}
&\frac{d}{dt}\left(\frac{1}{4}\norm{\ueb(t,\cdot)}^4_{L^{4}(\R)} +\frac{\left(A\beta\eps+B\beta\eps^2+C\beta^2\right)}{2}\norm{\pxx\ueb(t,\cdot)}^2_{L^2(\R)}       \right)\\
&\qquad\quad +\frac{d}{dt}\left(\frac{B\eps^2}{2}\norm{\px\ueb(t,\cdot)}^2_{L^2(\R)}+\frac{\left(B\beta^2\eps^2+C\beta^3\right)}{2}\norm{\pxxx\ueb(t,\cdot)}^2_{L^2(\R)}\right)\\ &\qquad\quad +\frac{C\beta^4}{2}\frac{d}{dt}\norm{\pxxxx\ueb(t,\cdot)}^2_{L^2(\R)}+\frac{\beta\eps A}{2}\norm{\ptx\ueb(t,\cdot)}^2_{L^2(\R)}\\
&\qquad\quad + \frac{3A\beta^2\eps}{8}\norm{\ptxx\ueb(t,\cdot)}^2_{L^2(\R)}+\frac{A\beta^3\eps}{2}\norm{\ptxxx\ueb(t,\cdot)}^2_{L^2(\R)}\\
&\qquad\quad +\beta^2\eps\left(\frac{C}{2} -2A\right)\norm{\pxxx\ueb(t,\cdot)}^2_{L^2(\R)}\\
&\qquad\quad +\left(\frac{5}{2} -2A-\frac{C_{0}D}{A} \right)\eps\norm{\ueb(t,\cdot)\px\ueb(t,\cdot)}^2_{L^2(\R)}\\
&\qquad\quad +\left(B-2B^2-D^2C^2_{0}\right)\eps^3\norm{\pxx\ueb(t,\cdot)}^2_{L^2(\R)}\\
&\qquad \le C_{0}\eps\norm{\px\ueb(t,\cdot)}^2_{L^2(\R)}.
\end{split}
\end{equation}
We search $A,\,B,\, C$ such that
\begin{equation*}
\begin{cases}
\displaystyle \frac{5}{2} -2A- \frac{C_0D}{A} >0,\\
\displaystyle B-2B^2 -D^2C^2_{0} >0,\\
\displaystyle \frac{C}{2} -2A>0,
\end{cases}
\end{equation*}
that is
\begin{equation}
\label{eq:p302}
\begin{cases}
\displaystyle 4A^2 - 5A +2C_{0} D<0,\\
\displaystyle 2B^2 -B -D^2C_{0}^2<0,\\
\displaystyle C>4A.
\end{cases}
\end{equation}
We choose
\begin{equation}
\label{eq:scet-C}
C=6A.
\end{equation}
The first inequality of \eqref{eq:p302} admits a solution, if
\begin{equation*}
25-32C_{0}D>0,
\end{equation*}
that is
\begin{equation}
\label{eq:cond-1}
D<\frac{25}{32C_0}.
\end{equation}
The second inequality of \eqref{eq:p302} admits a solution, if
\begin{equation*}
1-8D^2C^2_{0} >0,
\end{equation*}
that is
\begin{equation}
\label{eq:cond-2}
D<\frac{\sqrt{2}}{4C_{0}}.
\end{equation}
It follows from \eqref{eq:cond-1} and \eqref{eq:cond-2} that
\begin{equation}
\label{eq:sce-D}
D<\min\left\{\frac{25}{32C_0},\frac{\sqrt{2}}{4C_{0}}\right\}= \frac{\sqrt{2}}{4C_{0}}.
\end{equation}
Therefore, from \eqref{eq:p302}, \eqref{eq:scet-C} and \eqref{eq:sce-D}, we have that there exist $0<A_1<A_2$ and $0<B_1<B_2$, such that choosing
\begin{equation}
\label{eq:con-A-B-C}
A_1<A<A_2,\quad B_1<B<B_2,\quad C=6A,
\end{equation}
\eqref{eq:p302} holds.

\eqref{eq:p300} and \eqref{eq:p302} give
\begin{align*}
&\frac{d}{dt}\left(\frac{1}{4}\norm{\ueb(t,\cdot)}^4_{L^{4}(\R)} +\frac{\left(A\beta\eps+B\beta\eps^2+6A\beta^2\right)}{2}\norm{\pxx\ueb(t,\cdot)}^2_{L^2(\R)}       \right)\\
&\qquad\quad +\frac{d}{dt}\left(\frac{B\eps^2}{2}\norm{\px\ueb(t,\cdot)}^2_{L^2(\R)}+\frac{\left(B\beta^2\eps^2+6A\beta^3\right)}{2}\norm{\pxxx\ueb(t,\cdot)}^2_{L^2(\R)}\right)\\ &\qquad\quad +3A\beta^4\frac{d}{dt}\norm{\pxxxx\ueb(t,\cdot)}^2_{L^2(\R)}+\frac{\beta\eps A}{2}\norm{\ptx\ueb(t,\cdot)}^2_{L^2(\R)}\\
&\qquad\quad + \frac{3A\beta^2\eps}{8}\norm{\ptxx\ueb(t,\cdot)}^2_{L^2(\R)}+\frac{A\beta^3\eps}{2}\norm{\ptxxx\ueb(t,\cdot)}^2_{L^2(\R)}\\
&\qquad\quad +\beta^2\eps A\norm{\pxxx\ueb(t,\cdot)}^2_{L^2(\R)} +\eps K_1\norm{\ueb(t,\cdot)\px\ueb(t,\cdot)}^2_{L^2(\R)}\\
&\qquad\quad +\eps^3 K\norm{\pxx\ueb(t,\cdot)}^2_{L^2(\R)}\le C_{0}\eps\norm{\px\ueb(t,\cdot)}^2_{L^2(\R)},
\end{align*}
for some $K_1,\,K_2>0$.

\eqref{eq:u0eps}, \eqref{eq:l-2} and an integration on $(0,t)$ give
\begin{align*}
&\frac{1}{4}\norm{\ueb(t,\cdot)}^4_{L^{4}(\R)}+ \frac{\left(A\beta\eps+B\beta\eps^2+6A\beta^2\right)}{2}\norm{\pxx\ueb(t,\cdot)}^2_{L^2(\R)}\\
&\qquad\quad +\frac{B\eps^2}{2}\norm{\px\ueb(t,\cdot)}^2_{L^2(\R)}+\frac{\left(B\beta^2\eps^2+6A\beta^3\right)}{2}\norm{\pxxx\ueb(t,\cdot)}^2_{L^2(\R)}\\
&\qquad\quad +3A\beta^4\norm{\pxxxx\ueb(t,\cdot)}^2_{L^2(\R)}+\frac{\beta\eps A}{2}\int_{0}^{t}\norm{\ptx\ueb(s,\cdot)}^2_{L^2(\R)}ds\\
&\qquad\quad + \frac{3A\beta^2\eps}{8}\int_{0}^{t}\norm{\ptxx\ueb(s,\cdot)}^2_{L^2(\R)}ds+\frac{A\beta^3\eps}{2}\int_{0}^{t}\norm{\ptxxx\ueb(s,\cdot)}^2_{L^2(\R)}ds\\
&\qquad\quad +\beta^2\eps A\int_{0}^{t}\norm{\pxxx\ueb(s,\cdot)}^2_{L^2(\R)}ds +\eps K_1\int_{0}^{t}\norm{\ueb(s,\cdot)\px\ueb(t,\cdot)}^2_{L^2(\R)}ds\\
&\qquad\quad +\eps^3 K_2\int_{0}^{t}\norm{\pxx\ueb(s,\cdot)}^2_{L^2(\R)}ds\\
&\qquad \le C_{0} +C_{0}\eps\int_{0}^{t}\norm{\px\ueb(s,\cdot)}^2_{L^2(\R)}ds\le C_0.
\end{align*}
Hence,
\begin{align*}
\norm{\ueb(t,\cdot)}_{L^{4}(\R)}\le &C_{0}, \\
\eps\norm{\px\ueb(t,\cdot)}_{L^2(\R)}\le &C_{0},\\
\sqrt{\beta\eps}\norm{\pxx\ueb(t,\cdot)}_{L^2(\R)}\le & C_{0},\\
\eps\sqrt{\beta}\norm{\pxx\ueb(t,\cdot)}_{L^2(\R)}\le & C_{0},\\
\beta\norm{\pxx\ueb(t,\cdot)}_{L^2(\R)}\le & C_{0},\\
\beta\eps\norm{\pxxx\ueb(t,\cdot)}_{L^2(\R)}\le & C_{0},\\
\beta\sqrt{\beta}\norm{\pxxx\ueb(t,\cdot)}_{L^2(\R)}\le & C_{0},\\
\beta\norm{\pxxxx\ueb(t,\cdot)}_{L^2(\R)}\le & C_{0},\\
\beta\eps \int_{0}^{t}\norm{\ptx\ueb(s,\cdot)}^2_{L^2(\R)}ds \le&C_{0},\\
\beta^2\eps \int_{0}^{t}\norm{\ptxx\ueb(s,\cdot)}^2_{L^2(\R)}ds\le&C_{0},\\
\beta^3\eps\int_{0}^{t}\norm{\ptxxx\ueb(s,\cdot)}^2_{L^2(\R)}ds\le&C_{0},\\
\beta^2\eps\int_{0}^{t}\norm{\pxxx\ueb(s,\cdot)}^2_{L^2(\R)}ds\le &C_{0},\\
\eps\int_{0}^{t}\norm{\ueb(s,\cdot)\px\ueb(t,\cdot)}^2_{L^2(\R)}ds\le &C_{0},\\
\eps^3\int_{0}^{t}\norm{\pxx\ueb(s,\cdot)}^2_{L^2(\R)}ds\le &C_{0},
\end{align*}
for every $t>0$.

Arguing as \cite[Lemma $2.2$]{Cd}, we have \eqref{eq:ux-uxx} and \eqref{eq:uxx-l-2}.
\end{proof}
To prove Theorem \ref{th:main}. The following technical lemma is needed  \cite{Murat:Hneg}.
\begin{lemma}
\label{lm:1}
Let $\Omega$ be a bounded open subset of $
\R^2$. Suppose that the sequence $\{\mathcal
L_{n}\}_{n\in\mathbb{N}}$ of distributions is bounded in
$W^{-1,\infty}(\Omega)$. Suppose also that
\begin{equation*}
\mathcal L_{n}=\mathcal L_{1,n}+\mathcal L_{2,n},
\end{equation*}
where $\{\mathcal L_{1,n}\}_{n\in\mathbb{N}}$ lies in a
compact subset of $H^{-1}_{loc}(\Omega)$ and
$\{\mathcal L_{2,n}\}_{n\in\mathbb{N}}$ lies in a
bounded subset of $\mathcal{M}_{loc}(\Omega)$. Then $\{\mathcal
L_{n}\}_{n\in\mathbb{N}}$ lies in a compact subset of $H^{-1}_{loc}(\Omega)$.
\end{lemma}
Moreover, we consider the following definition.
\begin{definition}
A pair of functions $(\eta, q)$ is called an  entropy--entropy flux pair if $\eta :\R\to\R$ is a $C^2$ function and $q :\R\to\R$ is defined by
\begin{equation*}
q(u)=\int_{0}^{u} A\xi\eta'(\xi) d\xi.
\end{equation*}
An entropy-entropy flux pair $(\eta,\, q)$ is called  convex/compactly supported if, in addition, $\eta$ is convex/compactly supported.
\end{definition}
Following \cite{LN}, we prove Theorem \ref{th:main}.
\begin{proof}[Proof of Theorem \ref{th:main}]
Let us consider a compactly supported entropy--entropy flux pair $(\eta, q)$. Multiplying \eqref{eq:RKV-eps-beta} by $\eta'(\ueb)$, we have
\begin{align*}
\pt\eta(\ueb) + \px q(\ueb) =&\eps \eta'(\ueb) \pxx\ueb - \beta \eta'(\ueb) \pxxx\ueb\\
&-\beta\eta'(\ueb)\ptxx\ueb +\beta^2\eta'(\ueb)\ptxxxx\ueb \\
=& I_{1,\,\eps,\,\beta}+I_{2,\,\eps,\,\beta}+ I_{3,\,\eps,\,\beta} + I_{4,\,\eps,\,\beta}+I_{5,\,\eps,\,\beta}\\
&+I_{6,\,\eps,\,\beta}+I_{7,\,\eps,\,\beta}+I_{8,\,\eps,\,\beta},
\end{align*}
where
\begin{equation}
\begin{split}
\label{eq:12000}
I_{1,\,\eps,\,\beta}&=\px(\eps\eta'(\ueb)\px\ueb),\\
I_{2,\,\eps,\,\beta}&= -\eps\eta''(\ueb)(\px\ueb)^2,\\
I_{3,\,\eps,\,\beta}&= \px(-\beta\eta'(\ueb)\pxx\ueb),\\
I_{4,\,\eps,\,\beta}&= \beta\eta''(\ueb)\px\ueb\pxx\ueb,\\
I_{5,\,\eps,\,\beta}&= \px(-\beta \eta'(\ueb)\ptx\ueb),\\
I_{6,\,\eps,\,\beta}&= \beta\eta''(\ueb)\px\ueb\ptx\ueb,\\
I_{7,\,\eps,\,\beta}&=\px(\beta^2\eta'(\ueb)\ptxxx\ueb),\\
I_{8,\,\eps,\,\beta}&=-\beta^2\eta''(\ueb)\px\ueb\ptxxx\ueb.
\end{split}
\end{equation}
Fix $T>0$. Arguing as \cite[Lemma $3.2$]{Cd2}, we have that $I_{1,\,\eps,\,\beta}\to0$ in $H^{-1}((0,T) \times\R)$, and $\{I_{2,\,\eps,\,\beta}\}_{\eps,\beta >0}$ is bounded in $L^1((0,T)\times\R)$.\\
Arguing as \cite[Theorem $1.1$]{Cd}, we get $I_{3,\,\eps,\,\beta}\to 0$ in $H^{-1}((0,T) \times\R)$, and  $I_{4,\,\eps,\,\beta}\to 0$ in $L^1((0,T)\times\R)$.\\
We claim that
\begin{equation*}
I_{5,\,\eps,\,\beta}\to0 \quad \text{in $H^{-1}((0,T) \times\R),\,T>0,$ as $\eps\to 0$.}
\end{equation*}
By \eqref{eq:beta-eps} and Lemma  \ref{lm:ux-l-2},
\begin{align*}
&\norm{\beta  \eta'(\ueb)\ptx\ueb}^2_{L^2((0,T)\times\R)}\\
&\qquad \le\norm{\eta'}^2_{L^{\infty}(\R)}\beta^2\int_{0}^{T}\norm{\ptx\ueb(t,\cdot)}^2_{L^2(\R)}dt\\
&\qquad=\norm{\eta'}^2_{L^{\infty}(\R)}\frac{\beta^2\eps}{\eps}\int_{0}^{T}\norm{\ptx\ueb(t,\cdot)}^2_{L^2(\R)}dt\\
&\qquad\le C_{0}\norm{\eta'}^2_{L^{\infty}(\R)}\frac{\beta}{\eps}\le C_{0}\norm{\eta'}^2_{L^{\infty}(\R)}\eps^3\to0.
\end{align*}
We have that
\begin{equation*}
I_{6,\,\eps,\,\beta}\to0 \quad \text{in $L^1((0,T) \times\R),\,T>0,$ as $\eps\to 0$.}
\end{equation*}
Due to \eqref{eq:beta-eps}, Lemmas \ref{lm:l-2}, \ref{lm:ux-l-2} and the H\"older inequality,
\begin{align*}
&\norm{\beta\eta''(\ueb)\px\ueb\ptx\ueb}_{L^1((0,T)\times\R)}\\
&\qquad\le\norm{\eta''}_{L^{\infty}(\R)}\beta\int_{0}^{T}\!\!\!\int_{\R}\vert\px\ueb\ptx\ueb\vert dtdx\\
&\qquad\le\norm{\eta''}_{L^{\infty}(\R)}\frac{\beta\eps}{\eps}\norm{\px\ueb}_{L^2((0,T)\times\R)}\norm{\pt\ueb}_{L^2((0,T)\times\R)}\\
&\qquad \le C_{0}\norm{\eta''}_{L^{\infty}(\R)}\frac{\beta^{\frac{1}{2}}}{\eps} \le C_{0}\norm{\eta''}_{L^{\infty}(\R)}\eps\to0.
\end{align*}
We claim that
\begin{equation*}
I_{7,\,\eps,\,\beta}\to0 \quad \text{in $H^{-1}((0,T) \times\R),\,T>0,$ as $\eps\to 0$.}
\end{equation*}
By \eqref{eq:beta-eps} and Lemma \ref{lm:ux-l-2},
\begin{align*}
&\norm{ \beta^2\eta'(\ueb)\ptxxx\ueb}^2_{L^2((0,T)\times\R)}\\
&\qquad\le \beta^4 \norm{\eta'}_{L^{\infty}(\R)}\norm{\ptxxx\ueb}^2_{L^2((0,T)\times\R)}\\
&\qquad= \norm{\eta'}_{L^{\infty}(\R)}\frac{\beta^4\eps}{\eps}\norm{\ptxxx\ueb}^2_{L^2((0,T)\times\R)}\\
&\qquad\le C_{0}\norm{\eta'}_{L^{\infty}(\R)}\frac{\beta}{\eps} \le C_{0}\norm{\eta'}_{L^{\infty}(\R)}\eps^3\to0.
\end{align*}
We have that
\begin{equation*}
I_{8,\,\eps,\,\beta}\to0 \quad \text{in $L^1((0,T) \times\R),\,T>0,$ as $\eps\to 0$.}
\end{equation*}
Thanks to \eqref{eq:beta-eps}, Lemmas \ref{lm:l-2}, \ref{lm:ux-l-2} and the H\"older inequality,
\begin{align*}
&\norm{\beta^2\eta''(\ueb)\px\ueb\ptxxx\ueb}_{L^1((0,T)\times\R)}\\
&\qquad\le\beta^2\norm{\eta''}_{L^{\infty}(\R)}\int_{0}^{T}\!\!\!\int_{\R}\vert\px\ueb\ptxxx\ueb\vert dsdx\\
&\qquad\le\norm{\eta''}_{L^{\infty}(\R)}\frac{\beta^2\eps}{\eps}\norm{\px\ueb}_{L^2((0,T)\times\R)}\norm{\ptxxx\ueb}_{L^2((0,T)\times\R)}\\
&\qquad\le C_{0}\norm{\eta''}_{L^{\infty}(\R)}\frac{\beta^{\frac{1}{2}}}{\eps} \le C_{0}\norm{\eta''}_{L^{\infty}(\R)}\eps\to0.
\end{align*}
Therefore, \eqref{eq:con-u} follows from Lemma \ref{lm:1} and the $L^p$ compensated compactness of \cite{SC}.

To have \eqref{eq:entropy1}, we begin by proving that $u$ is a distributional solution of \eqref{eq:BU}.
Let $ \phi\in C^{\infty}(\R^2)$ be a test function with  compact support. We have to prove that
\begin{equation}
\label{eq:k1}
\int_{0}^{\infty}\!\!\!\!\!\int_{\R}\left(u\pt\phi+u^2\px\phi\right)dtdx +\int_{\R}u_{0}(x)\phi(0,x)dx=0.
\end{equation}
We define
\begin{equation}
u_{\eps_{n},\,\beta_{n}}:=u_n.
\end{equation}
We have that
\begin{align*}
\int_{0}^{\infty}\!\!&\!\!\!\int_{\R}\left(u_n\pt\phi+u^2_n\px\phi\right)dtdx +\int_{\R}u_{0,n}(x)\phi(0,x)dx\\
&+\eps_{n}\int_{0}^{\infty}\!\!\!\!\!\int_{\R}u_{n}\pxx\phi dtdx + \eps_n\int_{0}^{\infty}u_{0,n}(x)\pxx\phi(0,x)dx\\
&+ \beta_n\int_{0}^{\infty}\!\!\!\!\int_{\R}u_n\pxxx\phi dt dx + \beta_n\int_{0}^{\infty}u_{0,n}(x)\pxxx\phi(0,x)dx\\
&-\beta_n\int_{0}^{\infty}\!\!\!\!\!\int_{\R}u_n\ptxx\phi dtds - \beta_n\int_{0}^{\infty}u_{0,n}(x)\ptxx\phi(0,x)dx\\
&+\beta^2_{n}\int_{0}^{\infty}\!\!\!\!\!\int_{\R}u_n\ptxxxx\phi dtds -\beta_n\int_{0}^{\infty}u_{0,n}(x)\ptxxxx\phi(0,x)dx=0.
\end{align*}
Therefore, \eqref{eq:k1} follows from \eqref{eq:u0eps} and \eqref{eq:con-u}.

We conclude by proving that $u$ is the unique entropy solution of \eqref{eq:BU}. Fix $T>0$. Let us consider a compactly supported entropy--entropy flux pair $(\eta, q)$, and $\phi\in C^{\infty}_{c}((0,\infty)\times\R)$ a non--negative function. We have to prove that
\begin{equation}
\label{eq:u-entropy-solution}
\int_{0}^{\infty}\!\!\!\!\!\int_{\R}(\pt\eta(u)+ \px q(u))\phi dtdx\le0.
\end{equation}
We have
\begin{align*}
&\int_{0}^{\infty}\!\!\!\!\!\int_{\R}(\pt\eta(u_n)+\px q(u_n))\phi dtdx\\
&\qquad=\eps_{n}\int_{0}^{\infty}\!\!\!\!\!\int_{\R}\px(\eta'(u_n)\px u_n)\phi dtdx-\eps_{n}\int_{0}^{\infty}\!\!\!\!\!\int_{\R} \eta''(u_n)(\px u_n)^2\phi dtdx\\
&\qquad\quad -\beta_{n}\int_{0}^{\infty}\!\!\!\!\!\int_{\R}\px(\eta'(u_n)\pxx u_{n})\phi dtdx\\
&\qquad\quad+\beta_{n}\int_{0}^{\infty}\!\!\!\!\!\int_{\R}\eta''(u_n)\px u_n\pxx u_n\phi dtdx -\beta_{n}\int_{0}^{\infty}\!\!\!\!\!\int_{\R}\px(\eta'(u_n)\ptx u_n)\phi dtdx\\
&\qquad\quad +\beta_{n}\int_{0}^{\infty}\!\!\!\!\!\int_{\R}\eta''(u_n)\px u_n \ptx u_n\phi dtdx+\beta^2_{n}\int_{0}^{\infty}\!\!\!\!\!\int_{\R}\px(\eta'(u_n)\ptxxx u_n)\phi dtdx\\
&\qquad\quad -\beta^2_{n}\int_{0}^{\infty}\!\!\!\!\!\int_{\R}\eta''(u_n)\px u_n\ptxxx u_n\phi dtdx\\
&\qquad\le - \eps_{n}\int_{0}^{\infty}\!\!\!\!\!\int_{\R}\eta'(u_n)\px u_n\px\phi dtdx +\beta_{n}\int_{0}^{\infty}\!\!\!\!\!\int_{\R}\eta'(u_n)\pxx u_n\px\phi dtdx\\
&\qquad\quad+\beta_{n}\int_{0}^{\infty}\!\!\!\!\!\int_{\R}\eta''(u_n)\px u_n\pxx u_n\phi dtdx +\beta_{n}\int_{0}^{\infty}\!\!\!\!\!\int_{\R}\eta'(u_n)\ptx u_n\px\phi dtdx\\
&\qquad\quad +\beta_{n}\int_{0}^{\infty}\!\!\!\!\!\int_{\R}\eta''(u_n)\px u_n \ptx u_n\phi dtdx -\beta^2_{n} \int_{0}^{\infty}\!\!\!\!\!\int_{\R}\eta'(u_n)\ptxxx u_n\px\phi dtdx\\
&\qquad\quad  -\beta^2_{n}\int_{0}^{\infty}\!\!\!\!\!\int_{\R}\eta''(u_n)\px u_n\ptxxx u_n\phi dtdx\\
&\qquad \le \eps_{n}\int_{0}^{\infty}\!\!\!\!\!\int_{\R}\vert\eta'(u_n)\vert\vert\px u_n\vert\vert\px\phi\vert dtdx +\beta_{n}\int_{0}^{\infty}\!\!\!\!\!\int_{\R}\vert\eta'(u_n)\vert\vert\pxx u_n\vert\vert\px\phi\vert dtdx\\
&\qquad\quad +\beta_{n}\int_{0}^{\infty}\!\!\!\!\!\int_{\R}\vert\eta''(u_n)\vert\vert\px u_n\vert\vert\pxx u_n\vert\vert\phi\vert dtdx +\beta_{n}\int_{0}^{\infty}\!\!\!\!\!\int_{\R}\vert\eta'(u_n)\vert\vert\vert\ptx u_n\vert\vert\px\phi\vert dtdx\\
&\qquad\quad +\beta_{n}\int_{0}^{\infty}\!\!\!\!\!\int_{\R}\vert\eta''(u_n)\vert\vert\px u_n\vert\vert \ptx u_n\vert\vert\phi\vert dtdx +\beta^2_{n} \int_{0}^{\infty}\!\!\!\!\!\int_{\R}\vert\eta'(u_n)\vert\vert\ptxxx u_n\vert\vert\px\phi\vert dtdx\\
&\qquad\quad  +\beta^2_{n}\int_{0}^{\infty}\!\!\!\!\!\int_{\R}\vert\eta''(u_n)\vert\vert\px u_n\vert\vert\ptxxx u_n\vert\phi dtdx\\
&\qquad\le \eps_n\norm{\eta'}_{L^{\infty}(\R)}\norm{\px u_n}_{L^2(\supp(\px\phi))}\norm{\px \phi}_{L^2(\supp(\px\phi))}\\
&\qquad\quad+\beta_n\norm{\eta'}_{L^{\infty}(\R)}\norm{\pxx u_n}_{L^2(\supp(\px\phi))}\norm{\px \phi}_{L^2(\supp(\px\phi))}\\
&\qquad\quad+\beta_{n}\norm{\eta''}_{L^{\infty}(\R)}\norm{\phi}_{L^{\infty}(\R^{+}\times\R)}\norm{\px u_n\pxx u_n}_{L^1(\supp(\phi))}\\
&\qquad\quad+ \beta_{n}\norm{\eta'}_{L^{\infty}(\R)}\norm{\ptx u_n}_{L^2(\supp(\px\phi))}\norm{\px \phi}_{L^2(\supp(\px\phi))}\\
&\qquad\quad +\beta_{n}\norm{\eta''}_{L^{\infty}(\R)}\norm{\phi}_{L^{\infty}(\R^{+}\times\R}\norm{\px u_n\ptx u_n}_{L^1(\supp(\phi))}\\
&\qquad\quad+\beta^2_n\norm{\eta'}_{L^{\infty}(\R)}\norm{\ptxxx u_n}_{L^2(\supp(\px\phi))}\norm{\px \phi}_{L^2(\supp(\px\phi))}\\
&\qquad\quad  +\beta^2_{n}\norm{\eta'}_{L^{\infty}(\R)}\norm{\phi}_{L^{\infty}(\R^{+}\times\R)}\norm{\px u_n\ptxxx u_n}_{L^1(\supp(\phi))}\\
&\qquad \le \eps_n\norm{\eta'}_{L^{\infty}(\R)}\norm{\px u_n}_{L^2((0,T)\times\R)}\norm{\px \phi}_{L^2((0,T)\times\R)}\\
&\qquad\quad + \beta_n\norm{\eta'}_{L^{\infty}(\R)}\norm{\pxx u_n}_{L^2((0,T)\times\R)}\norm{\px \phi}_{L^2((0,T)\times\R)}\\
&\qquad\quad+\beta_{n}\norm{\eta''}_{L^{\infty}(\R)}\norm{\phi}_{L^{\infty}(\R^{+}\times\R)}\norm{\px u_n\pxx u_n}_{L^1((0,T)\times\R)}\\
&\qquad\quad+ \beta_{n}\norm{\eta'}_{L^{\infty}(\R)}\norm{\ptx u_n}_{L^2((0,T)\times\R)}\norm{\px \phi}_{L^2((0,T)\times\R)}\\
&\qquad\quad+ \beta_{n}\norm{\eta''}_{L^{\infty}(\R)}\norm{\phi}_{L^{\infty}(\R^{+}\times\R)}\norm{\px u_n\ptx u_n}_{L^1((0,T)\times\R)}\\
&\qquad\quad+\beta^2_n\norm{\eta'}_{L^{\infty}(\R)}\norm{\ptxxx u_n}_{L^2((0,T)\times\R)}\norm{\px \phi}_{L^2((0,T)\times\R)}\\
&\qquad\quad  +\beta^2_{n}\norm{\eta'}_{L^{\infty}(\R)}\norm{\phi}_{L^{\infty}(\R^{+}\times\R)}\norm{\px u_n\ptxxx u_n}_{L^1((0,T)\times\R)}.
\end{align*}
\eqref{eq:u-entropy-solution} follows from \eqref{eq:beta-eps}, \eqref{eq:con-u}, Lemmas \ref{lm:l-2} and \ref{lm:ux-l-2}.
\end{proof}

\section{The Rosenau-RLW equation}
\label{sec:Ro1}

In this section, we consider \eqref{eq:RKV34} and augment \eqref{eq:RKV34} with the initial condition
\begin{equation}
u(0,x)=u_{0}(x),
\end{equation}
on which we assume that
\begin{equation}
\label{eq:u-0-2}
u_0\in L^2(\R)\cap L^4(\R).
\end{equation}
We study the dispersion-diffusion limit for  \eqref{eq:RKV34}. Therefore, we consider the following fifth order problem
\begin{equation}
\label{eq:R-eps-beta}
\begin{cases}
\pt\ueb+ \px \ueb^2 - \beta\ptxx \ueb +\beta^2\ptxxxx\ueb=\eps\pxx\ueb, &\qquad t>0, \ x\in\R ,\\
\ueb(0,x)=u_{\eps,\beta,0}(x), &\qquad x\in\R,
\end{cases}
\end{equation}
where $u_{\eps,\beta,0}$ is a $C^\infty$ approximation of $u_{0}$ such that
\begin{equation}
\begin{split}
\label{eq:u0eps-1}
&u_{\eps,\,\beta,\,0} \to u_{0} \quad  \textrm{in $L^{p}_{loc}(\R)$, $1\le p < 4$, as $\eps,\,\beta \to 0$,}\\
&\norm{u_{\eps,\beta, 0}}^2_{L^2(\R)}+\norm{u_{\eps,\beta, 0}}^4_{L^4(\R)}+\beta\eps^2\norm{\pxx u_{\eps,\beta,0}}^2_{L^2(\R)}\le C_0,\quad \eps,\beta >0,
\end{split}
\end{equation}
and $C_0$ is a constant independent on $\eps$ and $\beta$.

The main result of this section is the following theorem.
\begin{theorem}
\label{th:main-1}
Assume that \eqref{eq:u-0-2} and  \eqref{eq:u0eps} hold.
If
\begin{equation}
\label{eq:beta-eps-2}
\beta=\mathbf{\mathcal{O}}(\eps^{4}),
\end{equation}
then, there exist two sequences $\{\eps_{n}\}_{n\in\N}$, $\{\beta_{n}\}_{n\in\N}$, with $\eps_n, \beta_n \to 0$, and a limit function
\begin{equation*}
u\in L^{\infty}(\R^{+}; L^2(\R)\cap L^{4}(\R)),
\end{equation*}
such that
\begin{align}
\label{eq:con-u-1}
&u_{\eps_n, \beta_n}\to u \quad \text{strongly in $L^{p}_{loc}(\R^{+}\times\R)$, for each $1\le p <4$};\\
\label{eq:entropy11}
&u \quad\text{is the unique entropy solution of \eqref{eq:BU}}.
\end{align}
\end{theorem}
Let us prove some a priori estimates on $\ueb$, denoting with $C_0$ the constants which
depend on the initial data.

We begin by observing that Lemma \ref{lm:l-2} holds also for \eqref{eq:R-eps-beta}.

Following \cite[Lemma $2.2$]{Cd}, or \cite[Lemma $2.2$]{CdREM}, or \cite[Lemma $4.2$]{CK}, we prove the following result.
\begin{lemma}\label{lm:ux-l-2-1}
Assume \eqref{eq:beta-eps-2}. For each $t>0$,
\begin{itemize}
\item[$i)$] the family $\{\ueb\}_{\eps,\beta}$ is bounded in $L^{\infty}(\R^{+};L^{4}(\R))$;
\item[$ii)$]the family $\{\eps\sqrt{\beta}\pxx\ueb\}_{\eps,\beta}$ is bounded in $L^{\infty}(\R^{+};L^2(\R))$;
\item[$iii)$] the families $\{\sqrt{\beta\eps}\ptx\ueb\}_{\eps,\beta},\,\{\beta\sqrt{\eps}\ptxx\ueb\}_{\eps,\beta},\,\{\beta\sqrt{\beta\eps}\ptxxx\ueb\}_{\eps,\beta},$\\
    $\{\sqrt{\eps}\ueb\px\ueb\}_{\eps,\beta}$ are bounded in\\ $L^2(\R^{+}\times\R)$.
\end{itemize}
\end{lemma}
\begin{proof}
Let $A$ be a positive constant which will be specified later. Multiplying  \eqref{eq:R-eps-beta} by $\ueb^3 -A\beta\eps\ptxx \ueb$, we have
\begin{equation}
\label{eq:l123}
\begin{split}
&\left(\ueb^3 -A\beta\eps\ptxx \ueb \right)\pt\ueb+2\left(\ueb^3 -A\beta\eps\ptxx \ueb \right)\ueb\px\ueb\\
&\qquad\quad-\beta\left(\ueb^3 -A\beta\eps\ptxx \ueb \right)\ptxx\ueb+\beta^2\left(\ueb^3 -A\beta\eps\ptxx \ueb \right)\ptxxxx\ueb\\
&\qquad=\eps\left(\ueb^3 -A\beta\eps\ptxx \ueb \right)\pxx\ueb.
\end{split}
\end{equation}
Since
\begin{align*}
\int_{\R}\left(\ueb^3 -A\beta\eps\ptxx \ueb \right)\pt\ueb dx=& \frac{1}{4}\frac{d}{dt}\norm{\ueb(t,\cdot)}^4_{L^4(\R)} + A\beta\eps\norm{\ptx\ueb(t,\cdot)}^2_{L^2(\R)},\\
2\int_{\R}\left(\ueb^3 -A\beta\eps\ptxx \ueb \right)\ueb\px\ueb dx =& -2A\beta\eps\int_{\R}\ueb\px\ueb\ptxx\ueb dx,\\
-\beta\int_{\R}\left(\ueb^3 -A\beta\eps\ptxx \ueb \right)\ptxx\ueb dx=& 3\beta\int_{\R}\ueb^2\px\ueb\ptx\ueb dx +A\beta^2\eps\norm{\ptxx\ueb(t,\cdot)}^2_{L^2(\R)}\\
\beta^2\int_{\R}\left(\ueb^3 -A\beta\eps\ptxx \ueb \right)\ptxxxx\ueb dx =& -3\beta^2\int_{\R}\ueb^2\px\ueb\ptxxx\ueb dx\\
&+ A\beta^3\eps\norm{\ptxxx\ueb(t,\cdot)}^2_{L^(\R)},\\
\eps\int_{\R}\left(\ueb^3 -A\beta\eps\ptxx \ueb \right)\pxx\ueb dx= & -3\eps\norm{\ueb(t,\cdot)\px\ueb(t,\cdot)}^2_{L^2(\R)}\\ &-\frac{A\beta\eps^2}{2}\frac{d}{dt}\norm{\pxx\ueb(t,\cdot)}^2_{L^2(\R)},
\end{align*}
integrating \eqref{eq:l123} on $\R$, we get
\begin{equation}
\label{eq:l589}
\begin{split}
&\frac{d}{dt}\left(\frac{1}{4}\norm{\ueb(t,\cdot)}^4_{L^4(\R)}+ \frac{A\beta\eps^2}{2}\norm{\pxx\ueb(t,\cdot)}^2_{L^2(\R)}\right)\\
&\qquad\quad +A\beta\eps\norm{\ptx\ueb(t,\cdot)}^2_{L^2(\R)}+A\beta^2\eps\norm{\ptxx\ueb(t,\cdot)}^2_{L^2(\R)}\\
&\qquad\quad + A\beta^3\eps\norm{\ptxxx\ueb(t,\cdot)}^2_{L^2(\R)} +3\eps\norm{\ueb(t,\cdot)\px\ueb(t,\cdot)}^2_{L^2(\R)}\\
&\qquad= 2A\beta\eps\int_{\R}\ueb\px\ueb\ptxx\ueb dx -3\beta\int_{\R}\ueb^2\px\ueb\ptx\ueb dx \\
&\qquad\quad  +3\beta^2\int_{\R}\ueb^2\px\ueb\ptxxx\ueb dx.
\end{split}
\end{equation}
It follows from \eqref{eq:You23}, \eqref{eq:beta-eps-1}, \eqref{eq:p289}, \eqref{eq:290} and \eqref{eq:l589} that
\begin{align*}
&\frac{d}{dt}\left(\frac{1}{4}\norm{\ueb(t,\cdot)}^4_{L^4(\R)}+ \frac{A\beta\eps^2}{2}\norm{\pxx\ueb(t,\cdot)}^2_{L^2(\R)}\right)\\
&\qquad\quad +\frac{A\beta\eps}{2}\norm{\ptx\ueb(t,\cdot)}^2_{L^2(\R)}+\frac{A\beta^2\eps}{2}\norm{\ptxx\ueb(t,\cdot)}^2_{L^2(\R)}\\
&\qquad\quad +\eps\left(3-\frac{2C_{0}D}{A}-2A\right)\norm{\ueb(t,\cdot)\px\ueb(t,\cdot)}^2_{L^2(\R)}\\
&\qquad\quad + \frac{A\beta^3\eps}{2}\norm{\ptxxx\ueb(t,\cdot)}^2_{L^2(\R)}\le 0.
\end{align*}
where $D$ is a positive constant which will be specified later.\\
We search a constant $A$ such that
\begin{equation*}
3-\frac{2C_{0}D}{A}-2A>0,
\end{equation*}
that is
\begin{equation}
\label{eq:l-A}
2A^2 -3A +2C_0 D <0.
\end{equation}
$A$ does exist if and only if
\begin{equation}
\label{eq:delta}
9-16C_{0}D>0.
\end{equation}
Choosing
\begin{equation}
\label{eq:d-3}
D= \frac{1}{16C_0},
\end{equation}
it follows from \eqref{eq:l-A} and \eqref{eq:delta} that there exist $0<A_1<A_2$, such that for every
\begin{equation}
\label{eq:A-8}
A_1<A<A_2,
\end{equation}
\eqref{eq:l-A} holds. Hence, we get
\begin{equation}
\label{eq:r13}
\begin{split}
&\frac{d}{dt}\left(\frac{1}{4}\norm{\ueb(t,\cdot)}^4_{L^4(\R)}+ \frac{A\beta\eps^2}{2}\norm{\pxx\ueb(t,\cdot)}^2_{L^2(\R)}\right)\\
&\qquad\quad +\frac{A\beta\eps}{2}\norm{\ptx\ueb(t,\cdot)}^2_{L^2(\R)}+\frac{A\beta^2\eps}{2}\norm{\ptxx\ueb(t,\cdot)}^2_{L^2(\R)}\\
&\qquad\quad +\eps K_1\norm{\ueb(t,\cdot)\px\ueb(t,\cdot)}^2_{L^2(\R)}+ \frac{A\beta^3\eps}{2}\norm{\ptxxx\ueb(t,\cdot)}^2_{L^2(\R)}\le 0.
\end{split}
\end{equation}
where $K_1$ is a fixed  positive constant. Integrating \eqref{eq:r13} on $(0,t)$, from \eqref{eq:u0eps-1}, we have
\begin{align*}
&\frac{1}{4}\norm{\ueb(t,\cdot)}^4_{L^4(\R)}+ \frac{A_3\beta\eps^2}{2}\norm{\pxx\ueb(t,\cdot)}^2_{L^2(\R)}\\
&\qquad\quad +\frac{A_3\beta\eps}{2}\int_{0}^{t}\norm{\ptx\ueb(s,\cdot)}^2_{L^2(\R)}ds+\frac{A_3\beta^2\eps}{2}\int_{0}^{t}\norm{\ptxx\ueb(s,\cdot)}^2_{L^2(\R)}ds\\
&\qquad\quad +\eps K_1\int_{0}^{t}\norm{\ueb(s,\cdot)\px\ueb(s,\cdot)}^2_{L^2(\R)}ds+ \frac{A_3\beta^3\eps}{2}\int_{0}^{t}\norm{\ptxxx\ueb(s,\cdot)}^2_{L^2(\R)}ds\le C_{0},
\end{align*}
Hence,
\begin{align*}
\norm{\ueb(t,\cdot)}_{L^4(\R)}\le &C_0,\\
\sqrt{\beta}\eps\norm{\pxx\ueb(t,\cdot)}_{L^2(\R)}\le &C_0,\\
\beta\eps\int_{0}^{t}\norm{\ptx\ueb(s,\cdot)}^2_{L^2(\R)}ds\le &C_0,\\
\beta^2\eps \int_{0}^{t}\norm{\ptxx\ueb(s,\cdot)}^2_{L^2(\R)}ds \le &C_0,\\
\eps\int_{0}^{t}\norm{\ueb(s,\cdot)\px\ueb(s,\cdot)}^2_{L^2(\R)}ds \le&C_0,\\
\beta^3\eps\int_{0}^{t}\norm{\ptxxx\ueb(s,\cdot)}^2_{L^2(\R)}ds\le& C_{0},
\end{align*}
for every $t>0$.
\end{proof}
Now, we are ready for the proof of Theorem \ref{th:main-1}.
\begin{proof}[Proof of Theorem \ref{th:main-1}]
Let us consider a compactly supported entropy--entropy flux pair $(\eta, q)$. Multiplying \eqref{eq:R-eps-beta} by $\eta'(\ueb)$, we have
\begin{align*}
\pt\eta(\ueb) + \px q(\ueb) =&\eps \eta'(\ueb) \pxx\ueb -\beta\eta'(\ueb)\ptxx\ueb +\beta^2\eta'(\ueb)\ptxxxx\ueb \\
=& I_{1,\,\eps,\,\beta}+I_{2,\,\eps,\,\beta}+ I_{3,\,\eps,\,\beta} + I_{4,\,\eps,\,\beta}+I_{5,\,\eps,\,\beta}+I_{6,\,\eps,\,\beta},
\end{align*}
where
\begin{equation}
\begin{split}
\label{eq:120004}
I_{1,\,\eps,\,\beta}&=\px(\eps\eta'(\ueb)\px\ueb),\\
I_{2,\,\eps,\,\beta}&= -\eps\eta''(\ueb)(\px\ueb)^2,\\
I_{3,\,\eps,\,\beta}&= \px(-\beta \eta'(\ueb)\ptx\ueb),\\
I_{4,\,\eps,\,\beta}&= \beta\eta''(\ueb)\px\ueb\ptx\ueb,\\
I_{5,\,\eps,\,\beta}&=\px(\beta^2\eta'(\ueb)\ptxxx\ueb),\\
I_{6,\,\eps,\,\beta}&=-\beta^2\eta''(\ueb)\px\ueb\ptxxx\ueb.
\end{split}
\end{equation}
Following Theorem \ref{th:main}, we have that $I_{1,\,\eps,\,\beta},\, I_{3,\,\eps,\,\beta},\, I_{5,\,\eps,\,\beta}\to 0$ in $H^{-1}((0,T)\times\R)$, $\{I_{2,\,\eps,\,\beta}\}_{\eps,\beta>0}$ is bounded in $L^1((0,T)\times\R)$,   
$I_{4,\,\eps,\,\beta},\,I_{6,\,\eps,\,\beta}\to 0$ in $L^1((0,T)\times\R)$. 

Arguing as Theorem \ref{th:main}, we get \eqref{eq:con-u-1} and \eqref{eq:entropy11}.
\end{proof}

\end{document}